\documentclass[12pt]{amsart}
\usepackage{amsmath,amsthm,amssymb,color,enumerate, tikz}
\usepackage{graphicx}
\usepackage[normalem]{ulem}
\usepackage{amsrefs}
\usepackage{hyperref}
\usepackage{float}
\usepackage{caption}
\usepackage{subcaption}
\usetikzlibrary{positioning}
\tikzset{main node/.style={circle,fill=blue!20,draw,minimum size=0.8cm,inner sep=0pt},
            }
\tikzset{duplicate node/.style={circle,fill=blue!40,draw,minimum size=0.8cm,inner sep=0pt},
            }

\usepackage{amscd}
\usepackage{longtable}
\usepackage{graphicx}
\usepackage{comment}
\usepackage{tikz-cd}
\hypersetup{colorlinks=true,linkcolor=blue,citecolor=blue}

\usepackage{mathtools}
\usetikzlibrary{cd}
\usetikzlibrary{decorations.pathmorphing}

\usepackage{amscd}

\newcommand{\soc}{\mathrm{Socle}}
\newcommand{\ma}{\mathrm{maxAss}}
\newcommand{\lcm}{\mathrm{lcm}}

\newtheorem{theorem}{Theorem}[section]
\newtheorem{proposition}[theorem]{Proposition}
\newtheorem{lemma}[theorem]{Lemma}
\newtheorem{corollary}[theorem]{Corollary}
\newtheorem{definition}[theorem]{Definition}
\newtheorem{remark}[theorem]{Remark}
\newtheorem{conjecture}[theorem]{Conjecture}
\newtheorem{example}[theorem]{Example}

\renewcommand{\caption}[1]{\singlespacing\hangcaption{#1}\normalspacing}

\title {Splittings for symbolic powers of edge ideals of complete graphs}

\author[S.~M.~Cooper]{Susan M. Cooper}
\address{Department of Mathematics\\
University of Manitoba\\
520 Machray Hall\\
186 Dysart Road\\
Winnipeg, MB\\
Canada R3T 2N2}
\email{susan.cooper@umanitoba.ca}

\author[S.~Da Silva]{Sergio Da Silva}
\address{Department of Mathematics and Economics\\
Virginia State University\\
1 Hayden Drive\\
Petersburg, VA\\
USA 23806}
\email{sdasilva@vsu.edu, smd322@cornell.edu}

\author[M.~Gutkin]{Max Gutkin}
\address{Department of Mathematics\\
University of Manitoba\\
420 Machray Hall\\
186 Dysart Road\\
Winnipeg, MB\\
Canada R3T 2N2}
\email{gutkinm4@myumanitoba.ca}

\author[T.~Reimer]{Tessa Reimer}
\address{Department of Statistics\\
University of Manitoba\\
330 Machray Hall\\
186 Dysart Road\\
Winnipeg, MB\\
Canada R3T 2N2}
\email{tessa.reimer@umanitoba.ca}

\keywords{Betti numbers, Splittings, Symbolic powers, Edge ideals}
\subjclass[2010]{13D02; 13F55}

\thanks{Cooper was supported by the Natural Sciences and Engineering Research Council of Canada (NSERC).  Da Silva was supported in part by a PIMS postdoctoral fellowship at the University of Manitoba and an NSERC Postdoctoral Fellowship at McMaster University.  Gutkin was supported by a University of Manitoba Undergraduate Research Award and Reimer was supported by a University of Manitoba Student Union (UMSU) Undergraduate Research Award.}

\begin{document}

\begin{abstract}
In this paper we study the $s$-th symbolic powers of the edge ideals of complete graphs.  In particular, we provide a criterion for finding an Eliahou-Kervaire splitting on these ideals, and use the splitting to provide a description for the graded Betti numbers.  We also discuss the symbolic powers and graded Betti numbers of edge ideals of parallelizations of finite simple graphs.
\end{abstract}\bigskip

\maketitle

\section{Introduction}
Edge ideals of finite simple graphs provide a rich collection of ideals for which many aspects of commutative algebra can be reduced to a combinatorial description (see \cite{Adam13}). For example, a primary decomposition of the edge ideal can be described using the facets of a simplicial complex via Stanley-Reisner theory. Graph-theoretic concepts such as graph colourings and vertex covers are also frequently used. It is natural to generalize this theory to other ideals which have similar combinatorial descriptions. Working with powers of square-free monomial ideals is one such direction (see \cite{Herzog-Hibi} for example). 

Another natural extension of this viewpoint is to consider symbolic powers of edge ideals of graphs.  This would serve a two-fold purpose of retaining some of the combinatorics enjoyed by edge ideals while also shedding light on properties of symbolic powers.  Indeed, symbolic powers have played a central role in many problems, yet their invariants are quite challenging to determine even when restricted to the family of square-free monomial ideals.  The Waldschmidt constant and resurgence number for certain cases were studied in \cite{GHOS18}, and questions about the symbolic Rees algebra were considered in \cite{Bahiano04}. In this article we study the graded Betti numbers for symbolic powers of edge ideals.  In particular, with a focus on complete graphs, we highlight the fruitfulness of using Eliahou-Kervaire splittings as a way to obtain information about the graded Betti numbers of a symbolic power of an ideal.

The symbolic power of a monomial ideal $I$ has a convenient description in terms of its primary decomposition, and is easier to describe compared to more general homogeneous ideals. Let $I\subset R = k[x_1, \ldots, x_m]$ be a monomial ideal, where $k$ is a field. Suppose that $I=I_1\cap\dots\cap I_r$ is a primary decomposition for $I$. Given a maximal associated prime ideal $Q$ of $I$ and a positive integer $s$, we define

\[I_{\subseteq Q} = \bigcap_{\sqrt{I_\ell}\subseteq Q}I_\ell\]

\noindent so that the $s$-th {\bf symbolic power} of $I$ is 

\[I^{(s)} = \bigcap_{Q\in\ma(I)}I_{\subseteq Q}^s,\]
where $\sqrt{I_\ell} = \{r \in R \mid r^n \in I_\ell \,\,\, \text{for some positive integer $n$}\}$ denotes the radical of $I_\ell$ and $\ma(I)$ denotes the set of associated primes of $I$ that are maximal with respect to inclusion.

This definition doesn't depend on the primary decomposition since $I_{\subseteq Q} =  R\cap IR_Q$ (see \cite{CEHH}). In Sections \ref{complete} and \ref{parallel} we provide a description of the minimal monomial generating set for $I(G)^{(s)}$ when $G$ is a complete graph or a parallelization of a finite simple graph. These descriptions are necessary to study Eliahou-Kervaire splittings for symbolic powers of edge ideals, which in turn allow us to compute graded Betti numbers recursively. Theorem \ref{further splittings} provides a criterion for defining an Eliahou-Kervaire splitting in this context. A similar technique is employed for a specific family of monomial ideals found in \cite{Valla05}.

Having a method to determine the graded Betti numbers of symbolic powers of edge ideals allows one to readily obtain information of related invariants of interest.  We demonstrate this for the minimum socle degree in Section \ref{socle}.  It is known that symbolic powers of edge ideals of complete graphs are Cohen-Macaulay and have dimension 1 (see Lemma 3.15 for the relevant citations).  Projectively these are fat points, and this implies that Section \ref{complete} is studying zero-dimensional arithmetically Cohen-Macaulay schemes, which complements the content found in \cite{Tohaneanu-VanTuyl}. 

Finally, a discussion about the graded Betti numbers for graph parallelizations of finite simple graphs can be found in Section \ref{parallel}. The definition of a graph parallelization is given there, but the utility of this construction comes from the ability to define families of graphs with properties that are related to the original graph (which can be useful for finding examples or counter-examples).

The arguments in this article work over any field $k$. While there exist some subtleties when working over fields of positive characteristic (for instance, \cite[Example 4.2]{FHV08} illustrates a Betti splitting which fails to be a splitting except in characteristic 2), our arguments rely on a particular splitting map which is characteristic-free.

\section{Preliminaries}

There is a rich theory involving edge ideals and their combinatorial properties.  We refer the reader to \cite{Adam13} for a thorough overview of this theory.  Throughout the paper, $G$ will denote an undirected finite simple graph with vertex set $V(G)=\{x_1,\ldots,x_m\}$ and edge set $E(G)$.  Recall that $G$ is {\bf simple} if it does not have multiple edges between vertices and if it does not have any loops at vertices.  Let $k$ be a field and $R=k[x_1,\ldots,x_m]$.  The {\bf edge ideal} of $G$ is defined as the square-free monomial ideal $I(G)=\langle x_ix_j: \{x_i,x_j\}\in E(G)\rangle \subset R$.

We mainly focus on complete graphs in this paper.  Recall that the {\bf complete graph} on $m$ vertices, denoted $K_m$, is the simple undirected graph for which every pair of distinct vertices is connected by exactly one unique edge.

In the sections that follow, we will also consider subgraphs of graphs. An \textbf{induced subgraph} $H$ of $G$ is a graph with vertex set $V(H)\subseteq V(G)$ and edge set $E(H)=E(G)\cap[V(H)]^2$. That is, if $u,v\in V(H)$, then $u$ and $v$ are adjacent in $H$ if and only if they are adjacent in $G$. For example, the complete graph $K_3$ is an induced subgraph of $K_4$, whereas the subgraph with 4 vertices but no edges is not.

Determining invariants of symbolic powers of homogeneous ideals in the polynomial ring can be quite challenging.  Recall from the introduction that we can define the $s$-th symbolic power of a monomial ideal $I$ in terms of a primary decomposition of $I$.  When $I$ is the edge ideal of a graph $G$, this becomes tractable via vertex covers of $G$.  A subset $W \subseteq V(G)$ is said to be a {\bf vertex cover} if $W \cap e \neq \emptyset$ for all $e \in E(G)$. A vertex cover $W$ is a {\bf minimal vertex cover} if no proper subset of $W$ is a vertex cover of $G$.  We have the following useful connection between vertex covers and $I(G)$.

\begin{lemma}[{\cite[Theorem 1.34, Corollary 1.35]{Adam13}}]\label{vertexcovers}
  Let $W_1, \ldots, W_t$ be the minimal vertex covers of a graph $G$.  If we set $\langle W_i \rangle = \langle x_j \mid x_j \in W_i \rangle$, then $I(G) = \langle W_1 \rangle \cap \cdots \cap \langle W_t \rangle$ is the minimal primary decomposition of $I(G)$.
\end{lemma}

One of the main goals of this paper is to investigate the use of a splitting technique due to Eliahou and Kervaire (in particular, a splitting for the symbolic powers of $I(G)$) to reduce the determination of graded Betti numbers to an induction on much simpler base cases.   We denote the set of minimal monomial generators of a monomial ideal $I\subset R$ (which are unique) by $\mathcal{G}(I)$, and recall the following definition from \cite{FHV08}.

\begin{definition}\label{EKdef}
	Let $I, J$ and $K$ be monomial ideals of $R = k[x_1,\ldots,x_m]$ such that $\mathcal{G}(I)$ is the disjoint union of $\mathcal{G}(J)$, $\mathcal{G}(K)$. We call $I=J+K$ an \textbf{Eliahou-Kervaire splitting} (or {\bf E-K splitting}) if there exists a splitting function $\mathcal{G}(J\cap K)\longrightarrow \mathcal{G}(J)\times \mathcal{G}(K)$ sending $w\mapsto (\phi(w), \varphi(w))$ such that:
	
	\begin{enumerate}
		\item $w=\lcm(\phi(w),\varphi(w))$; and
		\item for every subset $S\subset \mathcal{G}(J \cap K)$, both $\lcm(\phi(S))$ and $\lcm(\varphi(S))$ strictly divide $\lcm(S)$.
	\end{enumerate}
\end{definition}

\noindent One benefit of having an Eliahou-Kervaire splitting is the ability to compute the graded Betti numbers of $I$ in terms of the graded Betti numbers for $J$ and $K$. Recall that the {\bf $i,j$-th graded Betti number} of $I$ is by definition $\beta_{i,j}(I) = \dim_k\text{Tor}_i(k,I)_j$. That is, $\beta_{i,j}(I)$ is the number of
copies of $R(-j)$ appearing in the $i$-th module of the graded minimal free resolution of $I$:

$$0 \rightarrow \bigoplus_j R(-j)^{\beta_{\ell, j}(I)} \rightarrow \cdots \rightarrow \bigoplus_j R(-j)^{\beta_{1, j}(I)} \rightarrow \bigoplus_j R(-j)^{\beta_{0, j}(I)} \rightarrow I \rightarrow 0,$$
where $R(-j)$ is the polynomial ring $R$ shifted by degree $j$.

\begin{lemma}[{\cite[Proposition 3.2]{Fatabbi01}}]\label{EK_Betti}
Let $I, J$ and $K$ be monomial ideals of $k[x_1,\ldots,x_m]$ such that $I=J+K$ is an Eliahou-Kervaire splitting. Then 
\[\beta_{i,j}(I)=\beta_{i,j}(J)+\beta_{i,j}(K)+\beta_{i-1,j}(J\cap K),\]
for all $i\in\mathbb{N}$ and multidegrees $j$.
\end{lemma}

\begin{remark} 
The original version of this result was proved for total Betti numbers in \cite[Proposition 3.1]{EK90} over an arbitrary field. The proof of Lemma \ref{EK_Betti} is  also valid over any field, even if the author of \cite{Fatabbi01} assumes that the field is algebraically closed (which is needed later in \cite{Fatabbi01}).
\end{remark}

All E-K splittings are examples of Betti splittings, which are a choice of monomial ideals $I=J+K$ where $\mathcal{G}(I) = \mathcal{G}(J) \sqcup\mathcal{G}(K)$ which also satisfy the graded Betti number equality from Lemma \ref{EK_Betti}. The distinction will not be important for us since all of the splittings in the sections that follow are E-K splittings. See \cite{FHV08} for more information on Betti splittings. We finish with a simple auxiliary lemma, which we leave as an exercise.

\begin{lemma}\label{recursive}
If $I\subset k[x_1,\ldots,x_m]$ is a monomial ideal, then $\beta_{i,j}(x_{\ell}I)=\beta_{i,j-1}(I)$ for all $i,j\geq 1$, $1\leq \ell \leq m$. 
\end{lemma}

\section{Symbolic Powers Of Edge Ideals Of Complete Graphs}\label{complete}

We denote the complete graph on $m$ vertices by $K_m$ and label its vertices by $x_1, \ldots, x_m$.  Fix $R = k[x_1, \ldots, x_m]$.  Our main goal is to determine the graded Betti numbers for the symbolic powers of the edge ideal of $K_m$.  In order to make use of the Eliahou-Kervaire splitting technique, we need a convenient description for the minimal monomial generators of symbolic powers of $I(K_m)$.  The following lemma will simplify this task.

\begin{lemma}[{\cite[Lemma 2.6]{Wald15}}]\label{SymbolicGens}
Let $I \subset R = k[x_1,\ldots,x_m]$ be a square-free monomial ideal with minimal primary decomposition $I=P_1\cap\hspace{0.25mm}\cdots\hspace{0.25mm}\cap P_n$ with $P_{\ell}=\langle x_{j_1},\ldots,x_{j_{\alpha_{\ell}}}\rangle$ for $\ell=1,\ldots,n$. Then $x_1^{a_1}\cdots x_m^{a_m}\in I^{(s)}$ if and only if $a_{j_1}+\cdots +a_{j_{\alpha_{\ell}}}\geq s $ for $\ell=1,\ldots,n$.
\end{lemma}

We now determine the minimal monomial generating set $\mathcal{G}(I(K_m)^{(s)})$ for the $s$-th symbolic power of $I(K_m) \subset R$.

\begin{proposition} \label{minGens}
If $I=I(K_m) \subset R$ and $s \geq 2$, then $I^{(s)}$ has minimal monomial generating set 
$$\mathcal{L} :=\{x_1^{a_1}\cdots x_m^{a_m}: \exists \,\, 1 \leq i \leq m \,\, \text{with} \,\, \sum_{j \not = i}a_j = s, a_i = \max_{j \not = i} \{a_j\}\}.$$
That is, $\mathcal{G}(I^{(s)}) = \mathcal{L}$.
\end{proposition}

\begin{proof}
Let $x=x_1^{a_1}\cdots x_m^{a_m}\in \mathcal{L}$ where $a_i\geq 0$ for each $1 \leq i \leq m$. Without loss of generality, suppose that $a_1+\cdots +a_{m-1}=s$ and $a_m=\max\{a_1,\ldots,a_{m-1}\}$.  Note that any choice of $m-1$ vertices of $K_m$ defines a minimal vertex cover for $K_m$, and so by Lemma \ref{vertexcovers} we can write $I = \bigcap_{i=1}^m \langle x_1,\ldots,\hat{x_i},\ldots,x_m\rangle$,
which is the minimal primary decomposition for $I$. Then, by Lemma \ref{SymbolicGens}, $x\in I^{(s)}$ since every subset of $\{a_1,\ldots,a_m\}$ of size $m-1$ sums to a value larger than or equal to $s$.

Conversely, suppose that $x=x_1^{a_1}\cdots x_m^{a_m}\in I^{(s)}$. Without loss of generality, we may assume that $a_m=\max\{a_1,\ldots,a_m\}$ so that $a_1+\cdots+a_{m-1}\geq s$ by Lemma \ref{SymbolicGens}. Then $x$ is clearly in the ideal $\langle \mathcal{L} \rangle$, proving that $\mathcal{L}$ is a generating set for $I^{(s)}$.

To see that $\mathcal{L} = \mathcal{G}(I^{(s)})$, suppose $x=x_1^{a_1}\cdots x_m^{a_m}$ and $y=x_1^{b_1}\cdots x_m^{b_m}$ are both monomials in $\mathcal{L}$ such that $x$ divides $y$ (i.e. $b_t \geq a_t$ for $t=1,\ldots,m$). Suppose $1\leq i,j\leq m$ are such that
$$a_1+\cdots +a_m-a_i = b_1+\cdots +b_m-b_j=s \,\, \text{with} \,\, a_i=\max\{a_1,\ldots, a_m\}, \,\, b_j=\max\{b_1,\ldots,b_m\}.$$
Then $b_j\geq b_i\geq a_i$ and we can write $b_j -a_i = c \geq 0$. Thus,
$$b_1+\cdots +b_m = s+b_j = s+a_i+c = a_1+\cdots+a_m+c.$$

\noindent Therefore, 
$$c = \sum_{\ell=1}^m (b_{\ell} -a_{\ell}) =  (b_j-a_i) + (b_i-a_j) +  \sum_{\ell \neq i,j}^m (b_{\ell} -a_{\ell}) \implies  (b_i-a_j)+  \sum_{\ell \neq i,j}^m (b_{\ell} -a_{\ell}) = 0.$$

Note that since $a_i=\max\{a_1,\ldots,a_m\}$, we have that $a_j\leq a_i\leq b_i$ and so $b_i-a_j\geq 0$. Then each term in the previous equation is non-negative showing that $a_{\ell}=b_{\ell}$ for $\ell \neq i,j$ and $b_i=a_j$.  Since $b_j = \max\{b_1, \ldots, b_m\}$, there must be some $t \not = j$ such that $b_t = b_j$ (by the definition of $\mathcal{L}$).  Now if $b_j \not = b_i$, then $t \not = i$ and $a_t = b_t = b_j > b_i \geq a_i=\max\{a_1,\ldots, a_m\}$, a contradiction. Thus, $b_j = b_i = a_j$. We conclude that $x=y$, as required.
\end{proof}

\subsection{E-K Splittings}

We are now in a position to use splittings to determine the graded Betti numbers of symbolic powers of the edge ideal of a complete graph $K_m$.  As before, we let $R= k[x_1,\ldots,x_m]$.  Observe that if we set $G=K_m$ and fix $0\leq r \leq m$, then we can view $H=K_r$ as an induced subgraph of $K_m$ where $V(K_r) = \{x_1,\ldots, x_r\}$ (here $K_0$ is the null subgraph and $I(K_0)$ is the zero ideal). 

\begin{definition}
Let $G = K_m$ and $H = K_r$ for some fixed $0 \leq r \leq m$.  If $s \geq 2$ is an integer and $r \not = m$, then
\[ I_{H,s}   = \langle w\in \mathcal{G}(I(G)^{(s)}):x_i\nmid w, i=r+1,\ldots,m   \rangle \,\,\, \text{and} \,\,\, I_{G\setminus H,s} = I(G)^{(s)}\cap \langle \prod_{j=r+1}^m x_j\rangle.
\]
By convention, if $r=m$, then we define $I_{H,s} = I_{G\setminus H,s} = I(G)^{(s)}$.
\end{definition}

\noindent In general, if $w=x_1^{a_1}\cdots x_m^{a_m}\in \mathcal{G}(I(K_m)^{(s)})$, then $w\in \mathcal{G}(I_{H,s})$ if $a_{r+1}=\cdots=a_m=0$ and $w\in \mathcal{G}( I_{G\setminus H,s})$ if  $a_i\neq 0$ for all $i\in\{r+1,\ldots, m\}$.  Also, observe that $I_{H,s}$ can be viewed as the extension of $I(H)^{(s)} \subset k[x_1, \ldots, x_r]$ to the ring $k[x_1, \ldots, x_n]$.

\begin{lemma}\label{intersection}
If $m \geq 3, s\geq 2, G=K_m$ and $H=K_{m-1}$, then $I_{H,s}\cap I_{G\setminus H,s} = x_mI_{H,s}$.
\end{lemma}

\begin{proof}
	Note first that $I_{H,s}$ is an ideal contained in $I(G)^{(s)}$, and so $I_{H,s}\cap I(G)^{(s)}=I_{H,s}$.\\
	Thus, $I_{H,s}\cap I_{G\setminus H,s} = I_{H,s}\cap I(G)^{(s)}\cap \langle x_m\rangle = I_{H,s}\cap \langle x_m \rangle$.
	Since no generator of $I_{H,s}$ is divisible by $x_m$, it is clear that $I_{H,s}\cap \langle x_m \rangle=x_m I_{H,s}$.
\end{proof}

We now define an E-K splitting for ideals of the form $I_{K_m\setminus K_r,s}$. The ideal $I(K_m)^{(s)}$ is a special case, and an E-K splitting for it will follow as a corollary of the next theorem. These results are needed as part of the induction for the next section. 

\begin{theorem}\label{further splittings}
Let $m \geq 3, s\geq 2$ and $r \in \{1, \ldots, m \mid r \not = m-s-1\}$. If $G=K_m, H=K_{r}$, 
$$L_1=\langle w\in \mathcal{G}( I_{G\setminus H,s} ): x_{r}\mid w\rangle  =  I_{G\setminus H,s} \cap \langle x_{r}\rangle, \,\,\, \text{ and } \,\,\, L_2=\langle w\in \mathcal{G}( I_{G\setminus H,s} ):x_{r}\nmid w\rangle,$$
then $ I_{G\setminus H,s}=L_1+L_2$ is an Eliahou-Kervaire splitting.
\end{theorem}

\begin{proof}
We first note that the minimal monomial generating set of $L_2$ is given by all monomials $w= x_1^{a_1}\cdots x_m^{a_m}\in\mathcal{G}(I(G)^{(s)})$ such that $a_r=0$ and every exponent from $\{a_{r+1},\dots,a_m\}$ is nonzero.
	
To construct a splitting function that sends $w\in \mathcal{G}(L_1\cap L_2)$ to $(\phi(w),\varphi(w))\in \mathcal{G}(L_1)\times \mathcal{G}(L_2)$, we need to define two functions, $\varphi : \mathcal{G}(L_1\cap L_2)\rightarrow \mathcal{G}(L_2)$ and $\phi :\mathcal{G}(L_1\cap L_2)\rightarrow \mathcal{G}(L_1)$. We start with the function $\varphi$. First notice that by a similar argument to Lemma~\ref{intersection},  $L_1\cap L_2 = x_rL_2$ and each $w\in \mathcal{G}(L_1\cap L_2)$ can be written uniquely as $w=x_rv$, where $v\in \mathcal{G}(L_2)$. Therefore, we define $\varphi : \mathcal{G}(L_1\cap L_2)\rightarrow \mathcal{G}(L_2)$ by $w=x_rv\mapsto v.$ 

Next we define $\phi :\mathcal{G}(L_1\cap L_2)\rightarrow \mathcal{G}(L_1)$. Given $w=x_rv\in \mathcal{G}(L_1\cap L_2)$, let $a_i$ denote the exponent of $x_i$ in $w$ for $1\leq i\leq m$. Note that $a_r=1$.  Let $j \in \{1, \ldots, m\} \setminus \{r\}$ be the smallest index such that
$$a_1+\cdots +a_m - a_r -a_j=s, \hspace{2mm} \text{ and } \hspace{2mm} a_j=a_{max}\coloneqq \max (\{a_1,\ldots, a_m\}\setminus \{a_r,a_j\}).$$
Let $t \in \{1,\ldots, m\} \setminus \{j,r\}$ be the smallest index such that $a_t=a_{max}$.  The indices $j$ and $t$ exist by Proposition \ref{minGens}.  Observe that $j$ and $t$ are the two smallest indices in $\{1, \ldots, m\} \setminus \{r\}$ such that $a_t = a_j = a_{max}$ and that $j < t$.  Define $\phi:\mathcal{G}(L_1\cap L_2)\rightarrow \mathcal{G}(L_1)$ by 
$$w=x_rv=x_1^{a_1}\cdots x_m^{a_m}\mapsto \begin{cases}
\cfrac{w}{x_j} & \text{ if } \exists \hspace{0.9mm} \ell\neq j, t, r \text{ with } a_{\ell} = a_{max},\\
\cfrac{w}{x_tx_j} & \text{ otherwise.}
\end{cases}$$

We need to show that $\phi$ does in fact map into $\mathcal{G}(L_1)$. If $a_t$ is the only value in $\{a_1,\ldots,a_m\}\setminus \{a_j,a_r\}$ such that $a_t=a_{max}$, then let $A'=\{x_1,\ldots,x_m\}\setminus \{x_t\}$ and for each $1\leq i\leq m$, let $a_i'$ denote the exponent of $x_i$ in $\phi(w)$. Note that $a_j'=a_j-1, a_t'=a_t-1, a_r'=a_r=1$ and $a_l'=a_l$ for all $l\neq j,t$. Then
\begin{align*}
  \sum_{x_i\in A'}a_i'=\sum_{i=1}^m a_i'- a_t'= \deg(\phi(w))-(a_t-1) & =\deg(w)-2-(a_{max}-1)\\
  &=\sum_{i=1}^m a_i -1 - a_{max}\\
  &=\sum_{i=1}^m a_i -a_r-a_j=s.
\end{align*}
Also,
$$a_{max}'\coloneqq\max(\{a_1',\ldots,a_m'\}\setminus \{a_t'\})=a_j'= a_{max}-1 \ \ \text{and} \ \ a_t'=a_t-1=a_{max}-1=a_{max}'.$$
Thus $\phi(w)\in \mathcal{G}(I(G)^{(s)})$.

To show that $\phi(w) \in \mathcal{G}(I_{G \setminus H, s})$, it suffices to verify that $\prod_{i=r+1}^mx_i$ divides $\phi(w)$.  To this end, assume that $\prod_{i=r+1}^mx_i$ does not divide $\phi(w)$.  Recall that $w = x_rv$ for some $v \in \mathcal{G}(L_2)$.  Since $\mathcal{G}(L_2) \subset \mathcal{G}(I_{G \setminus H,s})$, we know that $\prod_{i=r+1}^m x_i$ divides $v$.  Thus, $\prod_{i=r+1}^mx_i$ divides $w$.  Since $\phi(w) = w/(x_jx_t)$ and $\prod_{i=r+1}^mx_i$ does not divide $\phi(w)$, it must be that at least one of $j$ or $t$ is greater than or equal to $r+1$ and $a_j = a_t = 1$.  Since $a_j=a_t = a_{max}$, this implies that $w$ is a square-free monomial.  Hence,
$$\deg(w) = a_1 + \cdots + a_m = s+2$$
and so $\phi(w)$ has degree $s$.  This is a contradiction since $\phi(w) \in \mathcal{G}(I(G)^{(s)})$ which consists of monomials of degree $s+1$ and higher.  We conclude that $\prod_{i=r+1}^m x_i$ divides $\phi(w)$, and thus $\phi(w) \in \mathcal{G}(I_{G \setminus H, s})$.  Furthermore, since $x_r\mid \phi(w)$, we cannot have $\phi(w)\in \mathcal{G}(L_2)$. Thus, $\phi(w) \in \mathcal{G}(L_1)$.

Similarly, if $a_t$ is not the only value in $\{a_1,\ldots, a_m\}\setminus \{a_j,a_r\}$ such that $a_t=a_{max}$, then define $A'$ and each $a_i'$ as before. Note that $a_j'=a_j-1, a_r=a_r'=1$ and $a_l'=a_l$ for all $l\neq j$. Then
$$\sum_{x_i\in A'} a_i'=\sum_{i=1}^ma_i'-a_t'=\deg(\phi(w))-a_t=\deg(w)-1-a_{max}=\sum_{i=1}^ma_i-a_r-a_j=s.$$
Also,
$$a_{max}'\coloneqq \max (\{a_1',\ldots,a_m'\}\setminus \{a_t'\})=a_{max} \ \ \ \text{and} \ \ \  a_t'=a_t=a_{max}=a_{max}'.$$
Thus, $\phi(w)\in \mathcal{G}(I(G)^{(s)})$.

We need to show that $\phi(w) = w/x_j \in \mathcal{G}(I_{G \setminus H, s})$.  Again, it suffices to verify that $\prod_{i=r+1}^mx_i$ divides $\phi(w)$.  Arguing by contradiction, suppose that $\prod_{i=r+1}^m x_i$ does not divide $\phi(w)$.  As above, $w = x_rv$ for some $v \in \mathcal{G}(L_2)$ and $\prod_{i=r+1}^m x_i$ divides $v$ and hence $w$.  Since $\phi(w) = w/x_j$, it follows that $j \geq r+1$ and $a_j = 1$.  Thus, since $t > j$ and $a_t = a_{max} = a_j$, we have $t > r+1$ and $1 = a_t = a_j = a_{max}$.  This implies that $w$ is a square-free monomial. Further, since the exponents in $\phi(w)$ of the variables in $A'$ sum to $s$ and $a_t'=a_t=1$, $\deg(\phi(w))=s+1$.  By the choice of $j$ and $t$, this implies that for all $1 \leq i < r$ we have $a_i \not = a_{max}$.  Since $w$ is square-free, we must have $a_i = 0$ for $1 \leq i < r$.  But then $w$ is square-free and divisible by $x_r$ and $\prod_{i=r+1}^mx_i$, and so $w = \prod_{i=r}^m x_i$.  Thus $\deg(\phi(w))=(m-r+1)-1=m-r$, and so
$$m-r=s+1 \implies m-s-1 = r,$$
a contradiction to the assumption that $m-s-1 \not = r$.  We conclude that $\prod_{i=r+1}^mx_i$ divides $\phi(w)$, and so $\phi(w) \in \mathcal{G}(I_{G \setminus H,s})$.  Again, since $x_r\mid \phi(w)$, we have that $\phi(w)$ is not in $\mathcal{G}(L_2)$ and thus $\phi(w) \in \mathcal{G}(L_1)$.

We now show that these maps define an E-K splitting. It is easy to verify the first condition by checking that $\lcm(\phi(w),\varphi(w))=w$ when $w\in \mathcal{G}(L_1\cap L_2)$. To check the second condition, we need to verify that for every subset $S \subset \mathcal{G}(J \cap K)$, both $\lcm(\phi(S))$ and $\lcm(\varphi(S))$ strictly divide $\lcm(S)$.  To this end, let $S\subset \mathcal{G}(L_1\cap L_2)$. Clearly, $\lcm(\phi(S))$ and $\lcm(\varphi(S))$ both divide $\lcm(S)$. Since $x_r\mid \lcm(S)$ and $x_r\nmid \lcm(\varphi(S))$, we cannot have $\lcm(\varphi(S))=\lcm(S)$, so $\lcm(\varphi(S))$ strictly divides $\lcm(S)$, as required.

To show $\lcm(\phi(S))\neq \lcm(S)$, let $a$ be the maximum exponent of any variable in any monomial in $S$. Let $i_0$ be the smallest index such that $x_{i_0}^a$ divides at least one $w\in S$. By definition of $\phi$, $x_{i_0}^a\nmid \phi(w)$. If there exists any $w'\in S$ such that $x_{i_0}^a\mid \phi(w')$, then either there exists an exponent in $w'$ that is larger than $a$, or there exists an index $i_1<i_0$ such that $x_{i_1}^a\mid w'$. Both are clear contradictions. Thus, $x_{i_0}^a\nmid \lcm(\phi(S))$, yet it clearly divides $\lcm(S)$, showing that $\lcm(\phi(S))\neq \lcm(S)$ and completing the proof.
\end{proof}

\begin{example}
The assumption that $r \not = m-s-1$ is necessary in Theorem~\ref{further splittings}.  For example, consider $m=5, s=2$ and $r=2$.  By definition, $x_3x_4x_5 \in \mathcal{G}(L_2)$, and so $w = x_2x_3x_4x_5 \in \mathcal{G}(L_1 \cap L_2)$.  Here, $a_{max} = 1$.  The smallest index $j \in \{1, \ldots, 5\} \setminus \{2\}$ such that $a_j = a_{max}$ is 3.  The smallest index $t \in \{1, \ldots, 5\} \setminus \{2,3\}$ such that $a_t = a_{max}$ is 4.  Note that $a_5 = a_{max}$ and $5 \not = j, t, r$.  Thus, $w/x_j = w/x_3 = x_2x_4x_5$ is not in $\mathcal{G}(I_{G \setminus H,s})$ as it is not divisible by $x_3x_4x_5$.
  
\end{example}

\begin{corollary}\label{splitting}
If $m \geq 3, s\geq 2, G=K_m$ and $H=K_{m-1}$, then $I(G)^{(s)} = I_{H,s} + I_{G\setminus H,s}$ is an E-K splitting.
\end{corollary}

\begin{proof}
  By definition, $I(K_m) = I(G)^{(s)} = I_{K_m \setminus K_m,s}$.  Therefore, by Theorem~\ref{further splittings} with $r=m$, we have that $I_{K_m/K_m, s} = L_1 + L_2$ is an E-K splitting where
  $$L_1 = I_{K_m \setminus K_m,s} \cap \langle x_m \rangle = I_{K_m \setminus K_{m-1}, s} \ \ \text{and} \ \ L_2 = I_{K_{m-1} \setminus K_{m-1},s} = I(K_{m-1})^{(s)} = I_{K_{m-1},s}.$$
\end{proof}

\begin{remark}
With these results, if $m - s - 1 < 0$, then we may repeatedly apply the splitting defined in Theorem~\ref{further splittings} to define a function that gives the graded Betti numbers of $I(G)^{(s)}$. In fact, it is an important observation that $L_1$ from Theorem~\ref{further splittings} is actually of the form $I_{K_{m}\setminus K_{r-1},s}$, since $I_{K_{m}\setminus K_{r-1},s} = I_{K_{m}\setminus K_r,s}\cap\langle x_r \rangle$, which allows the theorem to be iteratively applied to each subsequent $L_1$. Each step of this iteration applies Theorem~\ref{further splittings} to $I_{K_m\setminus K_{r},s}$ for decreasing $r$, and terminates with $I_{K_m\setminus K_0,s}$.

\end{remark}

\subsection{Graded Betti Numbers Of Symbolic Powers Of $I(K_2)$ and $I(K_3)$}

We now use splittings and induction to determine the graded Betti numbers for the $s$-th symbolic powers of the edge ideal of $K_3$.  We begin with the following observation.

\begin{lemma}\label{K2}
If $I = I(K_2) \subset R = k[x_1,x_2]$ and $s \geq 2$, then $\beta_{1,2s}(I_{K_2,s})=1$ and $\beta_{i,j}(I_{K_2,s})=0$ for $i \not = 1, j \not = 2s$.
\end{lemma}

\begin{proof}
Notice that $I^{(s)} = \langle x_1^sx_2^s\rangle = I^s$.  It is straightforward to see that a minimal graded free resolution of $I^{(s)}$ (over any field $k$) is given by $0 \rightarrow R(-2s) \rightarrow R \rightarrow  R/I^{(s)}\rightarrow 0$.
\end{proof}

\begin{theorem}\label{K3}
  If $i, j\in\mathbb{Z}^+$ and $s \geq 1$, then we have:
  $$\beta_{1,\frac{3s}{2}}(I(K_3)^{(s)}) = 1; \,\,\, \beta_{2, \frac{3s+3}{2}}(I(K_3)^{(s)}) = 2; \,\,\, \beta_{1,j}(I(K_3)^{(s)}) = 3 \, \,\,\, \text{if $\frac{3s+1}{2} \leq j \leq 2s$};$$
  $$\beta_{2,j}(I(K_3)^{(s)}) = 3 \, \,\,\, \text{if $\frac{3s+4}{2} \leq j \leq 2s+1$}; \,\,\, \text{and} \,\,\,\, \beta_{i,j}(I(K_3)^{(s)}) = 0 \,\,\,\, \text{otherwise}.$$
\end{theorem}

\begin{proof}
We induct on $s$. The result is clear for $s=1$ and $s=2$ via a computation using Macaulay2. Fix $s > 2$ and suppose the function holds for all positive integers $s' < s$. By Corollary \ref{splitting}, there is an E-K splitting of $I(K_3)^{(s)} = I_{K_2,s}+I_{K_3\setminus K_2,s}$, and by Lemma \ref{EK_Betti} we can write
\[\beta_{i,j}(I(K_3)^{(s)})=\beta_{i,j}(I_{K_2,s})+\beta_{i,j}(I_{K_3\setminus K_2,s})+\beta_{i-1,j}(I_{K_2,s}\cap I_{K_3\setminus K_2,s}).\]

Using Lemma \ref{intersection}, we know that $I_{K_2,s}\cap I_{K_3\setminus K_2,s} = x_3I_{K_2,s}$, and by Lemma \ref{recursive}, $\beta_{i-1,j}(x_3I_{K_2,s}) = \beta_{i-1,j-1}(I_{K_2,s})$. Therefore,
\[\beta_{i,j}(I(K_3)^{(s)})=\beta_{i,j}(I_{K_2,s})+\beta_{i,j}(I_{K_3\setminus K_2,s})+\beta_{i-1,j-1}(I_{K_2,s}).\]

We now write $\beta_{i,j}(I_{K_3\setminus K_2,s})$ in terms of the known graded Betti numbers coming from the induction.  Let $L_1$ and $L_2$ be as in Theorem \ref{further splittings} with $m=3$ and $r=2$. Then $I_{K_3\setminus K_2,s} = L_1+L_2$ is an E-K splitting, and therefore
\[\beta_{i,j}(I_{K_3\setminus K_2,s})=\beta_{i,j}(L_1)+\beta_{i,j}(L_2)+\beta_{i-1,j}(L_1\cap L_2).\]

From the proof of Theorem \ref{further splittings}, $L_1\cap L_2 = x_2L_2$ so that $\beta_{i-1,j}(L_1\cap L_2) = \beta_{i-1,j-1}(L_2)$. By a change of coordinates, we have that $\beta_{i,j}(L_2)=\beta_{i,j}(I_{K_2\setminus K_1,s})=\beta_{i,j}(I_{K_2,s})$. Therefore
\[\beta_{i,j}(I_{K_3\setminus K_2,s})=\beta_{i,j}(L_1)+\beta_{i,j}(I_{K_2,s})+\beta_{i-1,j-1}(I_{K_2,s}).\]

It remains to show that $\beta_{i,j}(L_1)$ can be computed using what is known from the induction.   We require one more E-K splitting. By definition, $L_1 = I_{K_3\setminus K_1,s}$, so we can apply Theorem \ref{further splittings} one more time (using $m=3$ and $r=1$) to get a splitting $L_1 = L_1'+L_2'$, where $L_1' = I(K_3)^{(s)}\cap\langle x_1x_2x_3\rangle$.  This yields
\[ \beta_{i,j}(L_1)=\beta_{i,j}(L_1')+\beta_{i,j}(L_2')+\beta_{i-1,j}(L_1'\cap L_2').   \] 

\noindent Using the same observations as before, notice that $\beta_{i,j}(L_2') = \beta_{i,j}(I_{K_2\setminus K_0,s})=\beta_{i,j}(I_{K_2,s})$ and $\beta_{i-1,j}(L_1'\cap L_2') = \beta_{i-1,j-1}(I_{K_2,s})$.
It is not difficult to see that $L_1' = x_1x_2x_3I(K_3)^{(s-2)}$ so that $\beta_{i,j}(L_1') = \beta_{i,j-3}(I(K_3)^{(s-2)})$. 

Overall we have shown that
\[\beta_{i,j}(I(K_3)^{(s)})=3\beta_{i,j}(I_{K_2,s})+3\beta_{i-1,j-1}(I_{K_2,s})+\beta_{i,j-3}(I(K_3)^{(s-2)}).\]

\noindent By Lemma \ref{K2}, we know that $\beta_{i,j}(I_{K_2,s})=1$ if $i=1$, $j=2s$ and $0$ otherwise. The result follows by induction.
\end{proof}

\begin{remark}
We relied on the reduction to $L_1' = x_1x_2x_3I(K_3)^{(s-2)}$ in the proof of Theorem \ref{K3}, and this is why 3 splittings are needed in the proof. See Lemma \ref{prod_int} for a generalization.
  \end{remark}

\subsection{Graded Betti Numbers Of Symbolic Powers Of $I(K_m)$ In General}

In Theorem \ref{K3}, we used the fact that $I(K_3)^{(s)}\cap\langle x_1x_2x_3\rangle = x_1x_2x_3I(K_3)^{(s-2)}$. The reader might wonder why we actually needed 3 splittings in the theorem. Perhaps the induction could be completed using just 2 splittings, for example, requiring a similar identification for $L_1= I(K_3)^{(s)}\cap\langle x_2x_3\rangle$. A look at the generators however shows that we need all of the variables present in the second ideal in the intersection to make such an identification. For example, we know that $\mathcal{G}(I(K_3)^{(3)})= \{x_1^3x_2^3, x_1^2x_2^2x_3, x_1^2x_2x_3^2, x_1x_2^2x_3^2,x_1^3x_3^3,x_2^3x_3^3\}$. If we look at the terms for which $x_2x_3$ can be factored out, we notice that the generator $x_1^2x_2^2x_3$ can be factored as $x_2x_3(x_1^2x_2)$, but $x_1^2x_2$ is not a minimal generator for  $I(K_2)^{(s)}$ or $I(K_3)^{(s)}$ with any choice of $s$.  We avoid this issue by only considering generators where each variable $x_i$ can be factored out. In particular, when computing graded Betti numbers for $I(K_m)^{(s)}$, one will need to use $m$ E-K splittings to reduce to the case $I_{K_m\setminus K_0,s}$ and achieve a similar result to Theorem \ref{K3}. 

\begin{lemma}\label{prod_int}
We have $I_{K_m\setminus K_0,s}= I(K_m)^{(s)} \cap \langle x_1 \cdots x_m \rangle = x_1 \cdots x_m I(K_m)^{(s-m+1)}$ when $s \geq m \geq 2$.
\end{lemma}

\begin{proof}
This follows from the observation that if $x_1^{a_1}\cdots x_m^{a_m} \in \mathcal{G}(I(K_m)^{(s)})\cap \langle x_1\cdots x_m\rangle$, then $x_1^{a_1-1}\cdots x_m^{a_m-1} \in \mathcal{G}(I(K_m)^{(s-m+1)})$.
\end{proof}

The underlying ideas in the proof of Theorem \ref{K3} can now be generalized to obtain formulae for the graded Betti numbers for the $s$-th symbolic powers of the edge ideal of $K_m$.  However, as $m$ increases, so does the complexity in writing the formulae.  For the sake of concreteness, we provide only the formulae for $m = 4$ below.   

\begin{theorem}
  If $i, j \in \mathbb{Z}^+$ and $s \geq 4$, then we have:
 $$\beta_{i,j}(I(K_4)^{(s)}) = 6 \,\,\,\, \text{if $i=1$, $j=2s$ or $i=3$, $j=2s+2$}; \,\,\, \beta_{2,2s+1}(I(K_4)^{(s)}) = 12; \,\,\, and$$
$$\beta_{i,j}(I(K_4)^{(s)}) = \beta_{i,j-4}(I(K_4)^{(s-3)}) + 4\beta_{i-1,j-4}(I(K_3)^{(s-2)}) + 4\beta_{i,j-3}(I(K_3)^{(s-2)}) \,\,\,\, \text{otherwise}.$$
\end{theorem}

While the statement of the formulae for these graded Betti numbers seems cumbersome, they are a direct result of an inductive computation as in Theorem \ref{K3}. 

\begin{remark}
Given any fixed $m>3$, we are able to inductively compute $\beta_{i,j}(I(K_m)^{(s)})$ for any $s>m-1$. If $s<m$, then at some point in the induction process $r=m-s-1$ and we cannot reduce the computation any further, and are left with finitely many Betti numbers to manually compute by other means.
  \end{remark}

\subsection{Minimum Socle Degree Of Symbolic Powers Of Edge Ideals Of Complete Graphs}\label{socle}

Having formulae for the graded Betti numbers of the symbolic powers for edge ideals of complete graphs gives us information on related invariants of the ideals.  For example, we can obtain the minimum socle degrees which we now discuss.  In the following, the dimension of a homogeneous ideal $I \subseteq R=k[x_1, \ldots, x_m]$ is the Krull dimension of $R/I$.  If $A=\oplus_{i\geq 0} A_i$ is a graded Artinian $k$-algebra, then $\mathfrak{m} = \oplus_{i\geq 1} A_i$ is a maximal ideal and we call the ideal quotient $\soc(A) = 0:\mathfrak{m} = \{r \in A \mid r\mathfrak{m} = 0\}$ the {\bf socle} of $A$. It is a finite-dimensional graded $k$-vector space, and we write $\soc(A) = \oplus_{i=0}k(-a_i)$ where $a_i\in\mathbb{Z}^+$ are the {\bf socle degrees} of $A$. The \textbf{minimum socle degree} of $A$ is the minimum of the $a_i$.

If $I$ is a homogeneous ideal of $R$ such that $R/I$ is Cohen-Macaulay, then we can find a maximal regular sequence $f_1,\ldots,f_n$ of $R/I$ where $f_i$ is a homogeneous polynomial of degree 1 and $n = \dim(R/I)$. Let $\bar{I} = I+\langle f_1,\ldots,f_n \rangle$. In this situation, the Artinian reduction $A=R/\bar{I}$ is $0$-dimensional, and hence Artinian. It is well-known that the socle degrees of $A$ are related to the back twists at the end of a minimal resolution of $R/I$.  More precisely, let us write the graded minimal free resolution for $R/I$ as

\[  0\rightarrow F_{m-1} = \bigoplus_i R(-a_i) \rightarrow \cdots \rightarrow F_1\rightarrow R\rightarrow R/I\rightarrow 0.\]

\vspace{2mm}

\noindent The last module in the free resolution for $R/\bar{I}$ is $F_{m-1}(-n)$. Since it is in position $m+n-1$ of the free resolution, the socle degrees of $A$ are $s_i= (a_i +n) - (m+n-1) = a_i -(m-1)$ by \cite[Lemma 1.3]{Kustin-Ulrich}. With a slight abuse of notation, we will say that socle degrees of $A$ are the socle degrees of $R/I$. In particular, the minimum socle degree of $R/I$ is just $\min_{i}\{s_i\}$. This is similar to the setup in \cite{Tohaneanu} and \cite{Tohaneanu-VanTuyl} except for the labelling of indices.

\begin{lemma}\label{CMcomplete}
If $I=I(K_m)\subset R =k[x_1,\ldots,x_m]$, then $R/I^{(s)}$ is Cohen-Macaulay of dimension 1.
\end{lemma}

\begin{proof}
The Cohen-Macaulay property follows from \cite[Theorem 3.6]{RTY12} and \cite[Theorem 2.1]{Varbaro}.  It suffices to show that $R/I$ has dimension 1 since  $I^{(s)}$ and $I$ have the same height. Each vertex cover of $K_m$ involves exactly $m-1$ vertices and defines a primary component in the primary decomposition of $I$ by Lemma \ref{vertexcovers}. Since $R/I$ is Cohen-Macaulay, all of the associated primes of $I$ must have the same height, so it suffices to compute the height of just one of these ideals. One such ideal is $\langle x_1,\ldots,x_{m-1}\rangle$ which has height $m-1$ (so $R/I$ has dimension 1), proving the result.
\end{proof}

As a consequence, we can apply our technique to determine the minimum socle degree of the $s$-th symbolic power of an edge ideal of any complete graph for any $s \geq 2$.  For example, the minimum socle degree of $R/I(K_2)^{(s)} = 2s-1$ and the minimum socle degree of $R/I(K_3)^{(3)} = 4$ since, by Theorem \ref{K3}, $\beta_{2,6}(I(K_3)^{(3)}) = 2, \beta_{1,5}(I(K_3)^{(3)}) = \beta_{1,6}(I(K_3)^{(3)}) = \beta_{2,7}(I(K_3)^{(3)}) = 3$ and $\beta_{i,j}(I(K_3)^{(3)}) = 0$ otherwise.

\section{Parallelizations}\label{parallel}

It is natural to ask if we can determine the graded Betti numbers of edge ideals of graphs obtained by certain graph operations.  One such operation is called a graph parallelization, which we now turn our attention to.  The notion of a graph parallelization appears in \cite[Section 2]{MMV11} in the discussion about polarizations and depolarizations of monomial ideals. We continue to work with undirected finite simple graphs.

\begin{definition}
  Let $G$ be a graph with vertex set $\{x_1, \dots, x_m\}$ and fix $\alpha = (\alpha_1, \dots, \alpha_m) \in (\mathbb{Z}^+)^m$. The \textbf{parallelization} of $G$ by $\alpha$, denoted $G^\alpha$, is the graph with vertex set $V(G^\alpha) = \{x_{1,1}, \dots, x_{1,\alpha_1}, \dots, x_{m,1},\dots, x_{m,\alpha_m}\}$ and edge set $E(G^\alpha) = \{\{x_{i,t},x_{j,\ell}\}|\{x_i,x_j\} \in E(G)\}$.
\end{definition}

For example, the graph for $K_3^{(3,1,1)}$ is obtained by duplicating $x_1$ to get the vertices $x_{1,1}, x_{1,2}$ and $x_{1,3}$ and adding edges between these vertices and $x_{2,1}$ and $x_{3,1}$.

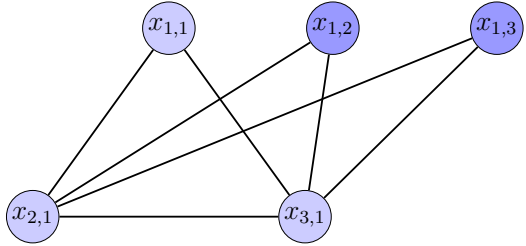
\begin{center}
\resizebox{7cm}{!}{
  \begin{tikzpicture}
    \node[main node] (1) {$x_{1,1}$};
    \node[main node] (2) [below left = 2.3cm and 1.5cm of 1]  {$x_{2,1}$};
    \node[main node] (3) [below right = 2.3cm and 1.5cm of 1] {$x_{3,1}$};
    \node[duplicate node] (4) [above right = 2.3cm and 4cm of 2] {$x_{1,2}$};
     \node[duplicate node] (5) [above right = 2.3cm and 6.5cm of 2] {$x_{1,3}$};

    \path[draw,thick]
    (1) edge node {} (2)
    (2) edge node {} (3)
    (3) edge node {} (1)
    (2) edge node {} (4)
    (3) edge node {} (4)
    (2) edge node {} (5)
    (3) edge node {} (5);
    
\end{tikzpicture}}
\end{center}

\noindent In particular, when $\alpha = (1,\ldots, 1)$, we recover the original graph so that $G^{(1,\ldots, 1)}$ is the same as $G$ (we are identifying $x_{i,1}=x_i$ in general). We will denote the set of vertices of $G^\alpha$ corresponding to the vertex $x_i \in V(G)$ by $V_i$. These are called the {\bf duplications} of $x_i$. The next lemma shows that all minimal vertex covers for $G^\alpha$ come from minimal vertex covers for $G$ replaced by the appropriate duplications $V_i$.  Recall that the {\bf open neighbourhood} of a given set $V'$ of vertices in a graph $G$, denoted $N(V')$, is the set of all vertices which are adjacent to vertices in $V'$, not including the vertices in $V'$ themselves.

\begin{lemma}\label{minVertexCov}
Let $G$ be a graph on $m$ vertices and fix $\alpha \in (\mathbb{Z^+})^{m}$. Then any minimal vertex cover for $G^\alpha$ has the form $\{x_{i_1,1},\dots, x_{i_1,\alpha_{i_1}},\dots  x_{i_{r},1} \dots, x_{i_{r},\alpha_{i_r}}\}$ where $\{x_{i_1}, \dots, x_{i_{r}}\}$ is a minimal vertex cover of $G$.
\end{lemma}

\begin{proof}
Let $S$ be a minimal vertex cover of $G$ and, without loss of generality, suppose that $S = \{x_1, \dots, x_n\}$. Let us define $S' = \{x_{1,1}, \dots, x_{1,\alpha_1}, \dots, x_{n,1}, \dots, x_{n,\alpha_n}\}$ as the set of vertices of $G^\alpha$ obtained by replacing each vertex $x_i$ in $S$ by all of its duplicates $V_i$ in $G^\alpha$. We first show that $S'$ is a minimal vertex cover of $G^\alpha$. 

Let $\{x_{i,t},x_{j,\ell}\}$ be any edge of $G^\alpha$. By definition, $\{x_i,x_j\} \in E(G)$. Since $S$ is a minimal vertex cover of $G$, at least one of $x_i$ or $x_j$ is in $S$. Without loss of generality, let us suppose that $x_i \in S$. Then $V_i\subset S'$, and so $S'$ contains a vertex from the edge $\{x_{i,t},x_{j,\ell}\}$. That is, $S'$ is a vertex cover of $G^\alpha$.

To see that $S'$ is minimal, suppose that there exists $x_{i,t} \in S'$ such that $S'\backslash\{x_{i,t}\}$ is a vertex cover of $G^\alpha$. Then necessarily $N(x_{i,t}) \subseteq S'$. By definition of a parallelization, $N(V_i) = N(x_{i,t})$. Hence, if $e \in E(G^\alpha)$ contains some vertex $x_{i,\ell}$, then $S'$ contains a vertex of $e$ other than $x_{i,\ell}$. That is, $S' \backslash \{V_i\}$ is a vertex cover of $G^\alpha$. However, since $N(V_i) \subseteq S'$, by construction of $S'$ it follows that $N(x_i) \subseteq S$. Then $S\backslash \{x_i\}$ is a vertex cover of $G$, contradicting the minimality of $S$.

It remains to show that these are the only minimal vertex covers of $G^\alpha$. Let $C '$ be a minimal vertex cover of $G^\alpha$. Note that for any $x_{i,t} \in C'$, there exists an edge $e \in E(G^\alpha)$ such that $e=\{x_{i,t},x_{j,\ell}\}$ (since $C'$ is a minimal vertex cover, $x_{i,t}$ cannot be an isolated vertex). However, $\{x_{i,t},x_{j,\ell}\} \in E(G^\alpha)$ if and only if $\{x_i,x_j\} \in E(G)$. So for all $x \in V_i, y \in V_j, \{x,y\} \in E(G^\alpha)$. Furthermore, if there exists some $V_i$ such that $V_i \nsubseteq C'$, then by the above, $N(V_i) \subseteq C'$. Hence, since $C'$ is minimal, $V_i \cap C' = \emptyset$ and so either $V_i \subseteq C'$ or $V_j \subseteq C'$. Thus $C' = V_1 \cup \dots \cup V_t$ for some labelling of the partite sets. Since $G$ is an induced subgraph of $G^\alpha$, $C'$ being a minimal vertex cover implies that $C = \{x_1, \dots, x_t\}= V(G)\cap C'$ is a vertex cover of $G$, and if $C$ were not minimal then the minimal vertex cover contained in $C$ would correspond to a minimal vertex cover of $G^{\alpha}$ properly contained in $C'$, a contradiction.
\end{proof}

\begin{proposition}\label{symbolicPowers}
Let $G$ be a graph with vertex set $\{x_1, \dots, x_m\}$ and edge ideal $I=I(G)$. Fix $\alpha = (\alpha_1, \dots, \alpha_m) \in (\mathbb{Z^+})^m$ and denote the edge ideal of $G^\alpha$ by $I_{\alpha}$. Then 
\[ \mathcal{G}(I_{\alpha}^{(s)}) = \{x_{1,1}^{e_{1,1}}\cdots x_{1,\alpha_1}^{e_{1,\alpha_1}}\cdots x_{m,1}^{e_{m,1}}\cdots x_{m,\alpha_m}^{e_{m,\alpha_m}} | \sum_{j=1}^{\alpha_i}e_{i,j} = a_i, \text{ where $x_1^{a_1}\cdots x_m^{a_m}\in \mathcal{G}(I)$}\}.\] 
\end{proposition}

\begin{proof}
The argument is similar to that of Proposition \ref{minGens}. We use Lemma \ref{SymbolicGens} and Lemma \ref{minVertexCov} to show that the set is generating. Minimality follows from the fact that $x_1^{a_1}\cdots x_m^{a_m}\in \mathcal{G}(I)$ and the description of minimal vertex covers.
\end{proof}

\subsection{Graded Betti Numbers Of Parallelizations}

The graded Betti numbers of parallelizations for complete graphs can be bounded below using the splittings from Corollary \ref{splitting}. This however is a special case of the next result.

\begin{proposition}\label{parallel_bound}
If $G$ is finite simple graph on $m$ vertices, $s \geq 2$, and $\alpha \in (\mathbb{Z^+})^{m}$, then for all $i, j \geq 1$, $\beta_{i,j}(I(G^\alpha)^{(s)}) \geq \beta_{i,j}(I(G)^{(s)})$. 
\end{proposition}

\begin{proof}
Since we can view $G$ as an induced subgraph of $G^\alpha$, the result follows as a direct consequence of \cite[Lemma 4.4]{GHOS18} (which is a generalization of the work in \cite{HH15} for edge ideals).
\end{proof}

Proposition \ref{parallel_bound} naturally leads one to try to determine what classes of ideals are obtained by parallelization of graphs $G$ where the Betti numbers of $I(G)^{(s)}$ are known.  We illustrate this useful direction with complete $n$-partite graphs.  Recall that the {\bf complete $n$-partite graph} with partite sets of size $a_1, \ldots, a_n$, denoted $K_{a_1,\ldots,a_n}$, is the graph with vertex set $V = \{x_{1,1},\ldots,x_{1,a_1},\dots,x_{n,1},\ldots,x_{n,a_n}\}$ and edge set $E = \{\{x_{i,j},x_{\ell,m}\}\:|\: i \neq \ell\}$.   In addition, recall that an {\bf independent set} of vertices of a graph $G$ is a set of vertices in which no two vertices are adjacent.

\begin{corollary}\label{complete n partite}
If $K_{a_1, \ldots, a_n}$ is a complete $n$-partite graph, then for all $s \geq 2$ and $i,j \geq 1$ we have the bound
  \[\beta_{i,j}(I(K_{a_1, \ldots, a_n})^{(s)}) \geq \beta_{i,j}(I(K_n)^{(s)}). \]
\end{corollary}

\begin{proof}
The result follows by noticing that $K_{a_1, \ldots, a_n} = K_n^{(a_1,\ldots, a_n)}$.  To see this, observe that the duplicates of each vertex in $K_n^{(a_1, \ldots, a_n)}$ form an independent set, and any two of these independent sets have all possible edges between them by definition of a parallelization applied to a complete graph.
\end{proof}

As a consequence, the results of Section 3 combined with Corollary \ref{complete n partite} yields explicit lower bounds for the graded Betti numbers of symbolic powers of edge ideals of complete $n$-partite graphs.  

Not surprisingly, the bound of Proposition \ref{parallel_bound} is not very effective as the entries in $\alpha$ grow. A better lower bound would not only depend on $I(G)^{(s)}$, but also on $\alpha$. 
As a motivational example, consider $I(K_3)^{(2)} = \langle \epsilon_1,\epsilon_2,\epsilon_3,\epsilon_4  \rangle=\langle x_1^2x_2^2, x_1^2x_3^2,x_2^2x_3^2, x_1x_2x_3 \rangle$.  One possible relation on these generators is $\sigma= x_3\epsilon_1 - x_1x_2\epsilon_4=0$. Now let $\alpha = (2,2,2)$ and consider $I(K_3^\alpha)^{(2)}$. Then the relation demonstrated by $\sigma$ also holds if we substitute any of the variable duplications. One possible choice is using $x_{3,2}$ instead of $x_{3,1}=x_3$. Now if $\epsilon_4' = x_1x_2x_{3,2}$, then $\sigma' = x_{3,2}\epsilon_1 - x_1x_2\epsilon_4'=0$ also (where $x_1 = x_{1,1}$ and $x_2 = x_{2,1}$). The same would hold true if we replaced any of the variables $x_i$ by one of their duplications $x_{i,j}$. In this way we get groupings of relations corresponding to some choice of the duplicated variables, which is also true for higher syzygies.

We see this more generally.  In particular, if $\sigma$ is a syzygy for the $i$-th step of the minimal graded free resolution of $I(G)^{(s)}$, then it is also a syzygy for the same step of the minimal graded free resolution of $I(G^\alpha)^{(s)}$, and substituting any variable in $\sigma$ with any of its duplicates also yields a syzygy.  We can expect $\prod_i \alpha_i$ many blocks of such relations, and so we might guess that $\beta_{i,j}(I(G^\alpha)^{(s)}) \geq (\prod_i \alpha_i)\beta_{i,j}(I(G)^{(s)})$.  Another lower bound would instead use blocks of relations coming from disjoint copies of $G$ in $G^{\alpha}$ (that is copies of $G$ which partition the vertex set of $G^{\alpha}$). There are exactly $\operatorname{min}\{\alpha_i\}$ many of these, leading to a weaker bound, but likely one which is easier to prove. Instead, we will conjecture that the stronger bound holds.  This bound works over any field since any potential cancellation in positive characteristic would occur both before and after any substitution by duplications.

\begin{conjecture}
Let $G$ be a graph on $m$ vertices and let $\alpha = (\alpha_1, \dots, \alpha_m) \in (\mathbb Z^+)^m$.  For all $i,j \in \mathbb{Z}^+$ and $s \geq 2$, we have $\beta_{i,j}(I(G^\alpha)^{(s)}) \geq (\prod_i\alpha_i)\beta_{i,j}(I(G)^{(s)})$.
\end{conjecture}

One might worry that the relations might not be part of a minimal generating set for the syzygy. However, we know from Gr\"obner theory that there at least $\binom{\prod_i\alpha_i}{2}$ relations which generate the syzygy, and the conjecture simply asks that a fraction of these are part of a minimal generating set. In particular, we know how to find generators for syzygies using $S$-polynomials (for example by Schreyer's Theorem). Information about $S$-polynomials and computing syzygies can be found in \cite{Eisenbud}.  To prove the conjecture, one would need to show that the $S$-polynomials corresponding to each choice of duplication are part of a minimal generating set.

\bibliographystyle{plain}
\bibliography{Library}

\end{document}